\documentclass[a4paper,10pt]{article}
\usepackage{amsmath}
\usepackage{amsfonts}
\usepackage{amssymb}
\usepackage{url}
\usepackage{color}
\usepackage{amssymb}
\usepackage{enumerate}

\def\endproof{\relax\ifmmode\expandafter\endproofmath\else
  \unskip\nobreak\hfil\penalty50\hskip.75em\hbox{}\nobreak\hfil\bull
  {\parfillskip=0pt \finalhyphendemerits=0 \bigbreak}\fi}
\def\endproofmath$${\eqno\bull$$\bigbreak}
\def\bull{\vbox{\hrule\hbox{\vrule\kern3pt\vbox{\kern6pt}\kern3pt\vrule}\hrule}}
\def\md#1{\ifmmode{\cal M}_\delta(#1)\else  % moduli space, delta decay of #1
{${\cal M}_\delta(#1)$}\fi}
\def\mb#1{\ifmmode{\cal M}_\delta^0(#1)\else  %moduli space, based, delta
{${\cal M}_\delta^0(#1)$}\fi}
\def\mdc#1#2{\ifmmode{\cal M}_{\delta,#1}(#2)\else    %moduli space, delta
{${\cal M}_{\delta,#1}(#2)$}\fi}
\def\mbc#1#2{\ifmmode{\cal M}_{\delta,#1}^0(#2)\else   %as before, based
{${\cal M}_{\delta,#1}^0(#2)$}\fi}
\def\mm{\ifmmode{\cal M}\else {${\cal M}$}\fi}

\def\msigma{\ifmmode{\cal M}^\sigma\else {${\cal M}^\sigma$}\fi}

\newtheorem{theorem}{Theorem}[subsection]
\newtheorem{proposition}[theorem]{Proposition}
\newtheorem{lemma}[theorem]{Lemma}

\newtheorem{corollary}[theorem]{Corollary}
\newtheorem{D}[theorem]{Definition}

\newtheorem{R}[theorem]{Remark}
\newenvironment{remark}{\begin{R}\rm }{\end{R}}

\def\ov{\overline}

\def\z2{{\mathbb Z}/{2{\mathbb Z}}}
\def\a{\alpha}

\def\L2{L^2(G)}			%L^2 functions on  a compact group G
\def\c{{\mathbb C}}			%complex space 
\def\z{{\mathbb Z}}			%integers
\def\p{{\mathbb P}}			%projective space
\def\x{\Delta_x}			%Defines the roots in $\Delta$ 
\begin{document}
\setlength{\parindent}{0pt} \setlength{\parskip}{2ex plus 0.4ex
minus 0.4ex}

\title{Centralizers of Commuting Elements in Compact Lie Groups}

\author{{\bf Kristen A. Nairn}\\College of St. Benedict, MN, USA\\
knairn@csbsju.edu\\
MSC Classes: 22C05 (Primary); 17B20, 22E46 (Secondary) }

\date{}
\maketitle
\begin{abstract} 
The moduli space for a flat $G$-bundle over the two-torus is completely determined by its holonomy representation. 
When $G$ is compact, connected, and simply connected, we show that the moduli space is homeomorphic to a product
of two tori mod the action of the Weyl group, or equivalently to the conjugacy classes of commuting pairs of elements in $G$. Since the component group for a non-simply connected group is given by some finite dimensional subgroup in 
the centralizer of an $n$-tuple, we use diagram automorphisms of the extended Dynkin diagram to prove properties of centralizers of pairs of elements in $G,$ followed by some explicit examples. We conclude by showing that as a result of a compact, connected, simply connected Lie group $G$ having a finite number of subgroups, each conjugate to the centralizer of any element in $G$, that there is a uniform bound on an irredundant chain of commuting elements. 
\end{abstract}

\section{Introduction}
Classifying the moduli space of gauge equivalence classes of flat connections on 
a principal $G-$bundle over a compact Riemann surface $\Sigma_g$ of genus $g$ is of 
interest from various perspectives. For example, Atiyah-Bott \cite{AB82} proved that this moduli
space is equivalent to the finite dimensional representation space 
$\{\rm{Hom}(\pi_1(\Sigma_g),G)\}/G$ by constructing a symplectic structure on the moduli space by symplectic reduction from the infinite-dimensional sympletic manifold of all connections. 

If $A$ is a flat connection on a $K-$bundle over $T^3$ then the holonomy of $A$ is defined by the 
conjugacy classes of commuting triples in $K.$ In topological field theory, vacua of Yang-Mills theory correspond to flat 
$G$ bundles. Every nontrivial triple on $T^3$ equals an additional quantum vacuum state and determines 
a distinct component of the moduli space $\mathcal{M}_G.$ If the triple
is of rank zero (rigid), then it is unique up to $G$-equivalence and 
every element can be conjugated into the maximal torus $T$ for $G.$ If 
the triple is not of rank zero, then there is an entire family of triples
with elements lying on some smaller torus inside the centralizer 
$Z_G(x,y,z).$ Kac-Smilga \cite{KS00} proved that computing the number of quantum vacuum states over $T^3$ is equivalent to classifying commuting triples in a simple, compact, simply connected Lie group $G.$ Witten \cite{W98} proved that the number of extra quantum vacuum states for a flat principal $G$-bundle over a 
spatial $3$-torus $T^3$ is the topological invariant called the {\sl Witten Index} which is equal to $g$, the dual Coxeter number of the Lie group $G.$ 

Our primary motivation comes from Borel-Friedman-Morgan \cite{BFM02}. Given $G$ is a compact, connected, semisimple Lie group, they proved that principal $G$-bundles $\zeta$ with flat connections over a maximal two torus $T^2$ are
classified up to restricted 
gauge equivalence by classifying commuting pairs of elements in the simply connected covering $\tilde{G}$ of $G$ that commute up to the center.  
The first invariant is the nontrivial characteristic 
class $[w]\in H^2(T^2,\pi_1(G))=\pi_1(G).$ By identifying $\pi_1(G)$ with a subgroup of the center
$\mathcal{C}G$ we fix a topological type of the bundle by $w(\zeta)=c\in\mathcal{C}G.$ Since the 
characteristic class is completely defined by the holonomy representation 
$\rho\colon \pi_1(T^2)\cong\z\times\z\to G$ where the images of $\rho$ 
commute then for any lifts $\tilde{x},\tilde{y}\in\tilde{G}$, we have $[\tilde{x},\tilde{y}]=[w]=c.$
Elements with this property are called $c-$pairs or "almost commuting".

\begin{D}
A pair of elements $x,y\in G$ {\bf commutes} if $[x,y]=1.$ 
A {\bf c-pair} in the simply connected covering $\tilde{G}$ 
of $G$ is a pair of elements $(x,y)$ where $x,y\in G$ such that $[x,y]=1$ and 
$[\tilde{x},\tilde{y}]=c\in\tilde{\mathcal{C}G}.$ 
\end{D} 

To understand why a flat bundle is determined by its holonomy representation note the following.
Let $G$ be a compact, connected and not necessarily simply connected Lie group and $\pi\colon \tilde{G}\to G$
the universal covering map. 
Certainly the choice of a lift $\tilde{x}$ is unique up to an element 
in $\rm{Ker}(\pi)\cong\pi_1(G)$ which is identified as a subgroup of the
center of the simply connected covering. Extending this for a 
$c$-pair: for $k\in\rm{Ker}(\pi),\ [\tilde{x},\tilde{y}]=
[k\tilde{x},k\tilde{y}]=c$ because $k\in \rm{Ker}(\pi)$ commutes with 
every element in $\tilde{G}$ and is also invariant under the choice
of $x,y.$ We may define conjugation by $\tilde{g}\in \tilde{G}$ to be
$\tilde{g}[\tilde{x},\tilde{y}]\tilde{g}^{-1}= 
[\tilde{g}\tilde{x}\tilde{g}^{-1},\tilde{g}\tilde{y}\tilde{g}^{-1}]$
satisfying $\pi(\tilde{g}\tilde{x}\tilde{g}^{-1})=g\pi(\tilde{x})g^{-1}=
gxg^{-1}.$ This lift is independent of the choice of  
$c\in\mathcal{C}\tilde{G}$ and thus our $c$-pair is well-defined.

For completeness, we recall some definitions found in \cite{Bourbaki} 
on Dynkin diagrams and root/coroot systems that we will use throughout the paper. 
Let $\Phi$ be a reduced irreducible root system for a compact connected Lie group $G,$ and 
let $\Delta=\{a_1,\ldots,a_n\}$ be a choice of simple roots for $G.$ 
Let $d$ be the highest root of $\Phi$ with respect to $\Delta.$ Set $\tilde{a}=-d$ and let 
$\tilde{\Delta}=\Delta\cup \{\tilde{a}\}$ be the extended set of simple roots. Then 
$\Delta^\vee$ is the set of coroots $a^\vee$ inverse to each root $a\in\Delta.$ 
If we define $A$ to be the unique alcove containing the origin in the positive Weyl chamber 
associated to $\Delta$ then there is a bijection between the walls of $A$ and $\tilde{\Delta}.$ 
Therefore $\tilde{\Delta}$ is the set of nodes for the extended Dynkin diagram $\tilde{D}(G).$
For each element $c\in\mathcal{C}G$ the differential $w_c\in \mathcal{W}$ of the action of the 
center on the alcove is a linear map normalizing $\tilde{\Delta}\subset t$ and the action of 
$w_c$ on the nodes of $\tilde{D}(G)$ is a diagram automorphism. Given 
a maximial torus $T\subset G$, denote Lie$(T)=\mathfrak{h}$ and the exponential map 
identifies $T$ with $\mathfrak{h}/Q^\vee$
where $Q^\vee=\sum\z a_i^\vee$ is the lattice associated to the 
coroots dual to a choice of simple roots $a_i\in\Delta$ for $G.$ Denote the affine Weyl group by $W_{aff}.$ 
The alcove is defined over the maximal torus $T\subseteq G$
as $A= \mathfrak{h}/W_{aff}(\Phi) \subseteq \mathfrak{h}$ where
$W_{aff}(\Phi)$ acts simply transitively on the set of alcoves in the vector space $V;$ thus
there is an induced action of the center $\mathcal{C}G$ on $A.$
x

\section{Component Group of the Centralizer of Commuting Pairs}
The work in \cite{BFM02} gave an explicit characterization of the moduli space of $c$-pairs in terms of the extended 
coroot diagram of a simply connected group $G$ and the action of the Weyl group on that diagram. 
This beneficial relationship between the root/coroot system and holonomy plays a crucial role.
Throughout, we assume that $G$ is a compact connected Lie group, unless otherwise defined. 
Note that if $G$ is disconnected, then there is the appearance of a 
$c-"1-$chain" coming from a component which may be a finite group of a certain order. 
A group $G$ is {\sl reductive} if any representation is irreducible. 
Notice that when a group is compact, it is equivalent to being reductive. 
We will use the following theorem by Borel (\cite{B62}, Theorem 5)
\begin{theorem} \label{Borel}
Let $G$ be a compact, connected and simply connected Lie group. Then 
$Z_G(x)$ is connected. 
\end{theorem}
The following demonstrates the relationship between the conjugacy classes 
for commuting pairs $(x,y)$ and flat $G$-bundles over $T^2:$
\begin{proposition}
Assume that $G$ is a compact, connected and simply connected Lie group.
For any maximal torus $S$, we have that
$\{\rm{Hom}(\pi_1(T^2,x),G)/G\}\to (S\times S)/\mathcal{W}$
is a homeomorphism. 
\end{proposition}
\begin{proof}
Fix generators $(\gamma_x,\gamma_y)$ for $\pi_1(T^2,x).$ 
Notice that we have a representation $\rho\colon\pi_1(T^2)\to G$ where
$\rho(\gamma_x)=x,\ \rho(\gamma_y)=y$ and that these images define the 
commutator in $G.$ In fact, the representation determines the commutator
in the following sense. Let $T$ be the maximal torus in $G.$ Then for 
some $g\in G,\ gxg^{-1}\in T$ and $gyg^{-1}\in T$ since every element 
in $G$ can be conjugated into the maximal torus. We 
want to show that both $x,y\in T.$ To do this, define conjugation by 
$g\in G$ for the pair $(x,y)$ by 
$g(x,y)g^{-1}=(gxg^{-1},gyg^{-1})=(x',y')$
where $(x,y)\in T\times T$ and $(x',y')\in T'\times T'.$  The fact that
$G$ is simply connected implies that $Z_G(x)$ is connected 
(~\ref{Borel}). Thus we may restrict to the connected component 
of the identity $Z^0(x).$ Since 
$x\in Z_G(x)$ we must show that $T\subseteq Z_G(x)$ because this
would imply that both $x,y\in T.$ By definition of the 
representation, the image  $[x,y]=1$ so that $y\in T$ is conjugate to 
$x$ which implies we may project $y$ to an element 
$\zeta_y\in\mathcal{W}.$ 
If we conjugate the pair $(x,y)$ by 
$(\zeta_y^{-1}g)\in G,$ then $(\zeta_y^{-1}g)(x,y)(\zeta_y^{-1}g)^{-1}=
\zeta_y^{-1}(x',y')\zeta\in T'\times T'.$ Thus we have shown that 
$(\zeta_y^{-1}g)$ conjugates elements from $T\times T$ to $T'\times T'$
and $\zeta_y^{-1}g\in\mathcal{W}.$ 

To prove the converse, notice that $\mathcal{W}$ acts by simultaneous 
conjugation
on $S\times S$ so that if $g\in \mathcal{W}$ is a reflection, then $gS\in 
N_G(S)/S.$
Thus $\mathcal{W}\times S\to S$ by $(gS,t)\mapsto gtg^{-1}$ and thus we 
have its
action on the pair $\mathcal{W}\times (S\times S)\to S\times S$ by 
$(gS,t,h)\mapsto (gtg^{-1},ghg^{-1}).$ Define the commutator by 
$[t,h]=[gtg^{-1},ghg^{-1}].$ Since elements in $S$ all commute,
if $\langle \gamma_x,\gamma_y\rangle$ generates $\pi_1(T^2)$ and 
$\rho(\gamma_x,\gamma_y)=[gxg^{-1},gyg^{-1}]=1$ then the holonomy 
determines the commutator and vice versa. 
\end{proof}

The next corollary follows immediately because the fundamental group of the centralizer $Z(x_1)$
is trivial, and for a commuting $n$-tuple in a simply connected group,
the component group is contained in the fundamental group of the semisimple subgroup 
$\pi_1(DZ(x_1)).$
\begin{corollary}
When $G$ is of type $A_n,\ C_n$ every commuting $n$-tuple can be
conjugated into the maximal torus $T$ in $G$ so that the moduli
space has the form $\mathfrak{M}_G = 
\underset{n}{\underbrace{T\times\cdots\times T}}/\mathcal{W}.$
\end{corollary} 
The corollary can also be seen directly as follows. If 
$(x_1,\ldots,x_n)$ is a commuting $n$-tuple such that 
$[\tilde{x_1},\tilde{x_2}]=c\in \mathcal{C}G$ and 
$[\tilde{x_1},\tilde{x_i}]=1,\ \forall\  2<i\le n,$ choosing 
$\tilde{x_1}$ in the alcove over the torus implies that 
$\tilde{x_2}$ projects to a Weyl element and therefore conjugates 
back into the maximal torus; every other element has trivial 
commutator and thus can be conjugated to the maximal torus.
This also works when $(x_i,x_j)$ 
for $1\le i<j\le n$ is an arbritrary $n$-tuple because the lifts 
$[\tilde{x_i},\tilde{x_j}]=c_{ij}\in\mathcal{C}G$ and for the cases
of type $A_n,\ C_n$ the center is generated by one cyclic element.
Hence only one pair in the $n$-tuple determines what happens to the
other elements in the $n$-tuple. 

\begin{corollary}
When $G$ is simply connected, the component group $\pi_0(Z(S))$ is a 
subgroup of $\z/{n_i}\z$ where $n_i\le 6$
and corresponds to the coroot integer for $x_1\in G$ which is 
associated to the node in the extended Dynkin diagram $\tilde{D}(G).$
\end{corollary}

By [Lemma 3.1.5 in \cite{BFM02}], for $x_n\in Z(x_1,\ldots,x_{n-1}),\   
{\rm Stab}_{W(\Phi(x_1,\ldots,x_{n-1}))}(x_n)\cong
{\rm Stab}_{\pi_1(DZ(x_1,\ldots,x_{n-1}))}(\tilde{x_n})).$
Thus the component group is a subgroup of the fundamental group
$\pi_1(Z(x_1,\ldots,x_{n-1})),$ which in turn is a finite subgroup
of the center of $\tilde{Z}(x_1,\ldots,x_{n-1}).$ If the
fundamental group of the centralizer $Z(x_1)$ is trivial, then
there are no further component groups for the remaining elements
in the $n$-tuple. So suppose that $\pi_1(DZ(x_1))\ne \{1\}.$
If $G$ is simply connected, then  $Z(x)$ is
connected and thus $\pi_0(Z(x_1,x_2))\subseteq \pi_1(DZ(x_1)).$
Even if $G$ is not simply connected but still connected, choosing
$x_2$ to lie in the connected component of the centralizer of 
$Z(x_1)$ will yield the same result.

If $\pi_1(DZ(x_1))$ is not central then it defines a possibly 
non-trivial diagonal subgroup $\mathcal{C}\tilde{Z}(x_1)$ which acts 
simultaneously on the components of the centralizer $Z(x_1).$  
If $x_2$ is chosen so that $\tilde{x_2}\in A$ does not lie in the
fixed point space under this diagonal action, the component
group is trivial. However, if $\tilde{x_2}$ lies somewhere in the
fixed space under the action of the center, there will be a
nontrivial component group $\pi_0(Z(x_1,x_2)).$ The structure of
the centralizer itself will differ depending on where the element
$x_2$ lies. Regardless of which compact, connected Lie group $G$
we are working with, whether or not there is a component group   
from the choice of second element relies on the diagonal action 
on $Z(x_1,x_2).$

\begin{proposition}
Let $G$ be simply connected and let $\Delta=\{a_1,\ldots,a_n\}$ be
a choice of simple roots. Let $\x=\{\tilde{a},a_1,\ldots,a_k\}, k\le n,$
be a choice of simple roots for $Z_G(x)$ and let  
$\mathfrak{h}(x)\subseteq\mathfrak{h}$ be the real linear
span of the coroots dual to the roots in $\x.$ Then there  
is an exact sequence $1\to Q^\vee(x)\to Q^\vee\cap\mathfrak{h}(x)\to\z/n_i\z \to 1.$
\end{proposition}  

\begin{proof}
By definition of the fundamental group,
$n_i={\rm gcd}(g_{k+1},\ldots,g_n)$ knowing that
all the coroot integers for both the classical and exceptional
groups are less than or equal to six, $n_i\le 6.$ Dividing each of
the coroot integers in any group $G$ by $n_i$ we may define a new
integer $g_r'=g_r/{n_i}$ for $r > k.$ By definition, this element
will have order $n_i$ in the central subgroup
$Q^\vee\cap\mathfrak{h}/{Q^\vee(x)}$ and thus is a generator for the
cokernel. Hence
$Q^\vee\cap\mathfrak{h}/{Q^\vee(x)}\cong\z/n_i\z.$
\end{proof}

\begin{proposition}
For an arbitrary compact, connected simple group $G$ and for a 
commuting $n$-tuple $\ov{x}=(x_1,\ldots,x_n),$ the component group of 
the centralizer of the $n$-tuple can be defined in terms of the
roots as 
$$\pi_0(Z_G(\ov{x}))=\frac{\rm{Stab}_{L/{Q^\vee}}(\ov{x})}
{\mathcal{W}(\Phi(\ov{x}))}.$$ 
\end{proposition}
\begin{proof}
If $G$ is not simply connected, then under complexification 
$T_{\c}=\mathfrak{h}/L$, where $Q^\vee\subseteq L\subseteq P^\vee.$
If $\Phi(x)$ is a subset of roots which annihilate $x$ we must 
determine how $\rm{Stab}_{\mathcal{W}}(x)$ is defined with respect
to this smaller subset of roots. If $x$ corresponds to some node $a_i$
in the extended Dynkin diagram such that $g_{a_i}\ne 1$ then 
$\Phi(x)=\{a_k\in\tilde{\Delta}\mid k\ne 1\}.$ Since 
$Q^\vee\subseteq L\subseteq P^\vee$ we have the nesting of tori
$\mathfrak{h}/{P^\vee}\subseteq \mathfrak{h}/L \subseteq 
\mathfrak{h}/{Q^\vee}.$ Thus for the lift 
$\mathfrak{h}/L\to \mathfrak{h}/{Q^\vee}$ sending $x\mapsto \tilde{x}$ 
its kernel consists of all the roots in $L$ not in $Q^\vee$ i.e. 
$L/{Q^\vee}.$ Therefore the roots which annihilate $x$ are the same
as those annihilating $\tilde{x}.$ 

Let $\mathcal{W}(\Phi(x))$ be a subgroup in $\mathcal{W}$ defined 
by a subroot system when viewed as characters which annihilate $x.$ 
The faithful action of $\mathcal{W}_{aff}$ on $\mathfrak{h}/{Q^\vee}$
yields a split exact sequence 
$1\to P^\vee/{Q^\vee}\to \mathcal{W}_{aff}\to \mathcal{W}\to 1$
and since the kernel is central, 
$\mathcal{W}_{aff}=P^\vee/{Q^\vee}\times \mathcal{W}$ is a direct product
because the action of the Weyl group is trivial on the center. 
Restrict the Weyl group to $L\colon \ \mathcal{W}^L_{aff}=L/{Q^\vee}\times 
\mathcal{W}.$ The torus action of $L/{Q^\vee}$ on $\mathfrak{h}/{Q^\vee}$
provides the quotient 
$(\mathfrak{h}/{Q^\vee})/(L/{Q^\vee})=\mathfrak{h}/L.$  This implies that
$\rm{Stab}_{\mathcal{W}}(x)\subseteq L/{Q^\vee}\times\mathcal{W}$ and the 
projection $\pi\colon 
(\mathfrak{h}/{Q^\vee})\to (L/{Q^\vee})$ satisfies $\pi^{-1}(\tilde{x})=x$
as the unique lift to the alcove. Therefore we may define 
$\rm{Stab}_{L/{Q^\vee}}(x)=\rm{Stab}_{\mathcal{W}_{aff}}(\tilde{x})$ 
in the sense that the roots which annihilate $\tilde{x}$ can be used 
to define a subset $S\subseteq L/{Q^\vee},$ where
$S=\pi_0(Z_G(x_1,\ldots,x_n)).$ This allows for 
a component group larger than the fundamental group and therefore it is 
not necessarily cyclic. Since $S\subseteq \mathcal{C}G$ it induces a 
well-defined cyclic permutation on the vertices in the alcove and its
fixed space $\mathfrak{h}^S$ may be something other than the barycenter.

$L$ is defined as follows. For $Z_G(x_1)$
the vector space is $\mathfrak{h}$ and its coroot lattice is the entire
$Q^\vee.$ Because we are considering commuting elements, we choose 
$x_2\in Z_G(x_1)$ to lie in $\mathfrak{h}(x_1)/L_1$ where 
$\mathfrak{h}(x_1)$ is the vector space associated to $DZ(x_1).$ Because 
$Z(x_1)$ is not necessarily connected, the associated lattice is
$Q^\vee\subseteq L_1 \subseteq P^\vee.$ By induction,  the element 
$x_n\in \mathfrak{h}(x_1\ldots,x_{n-1})/L_{n-1}$ and 
$Q^\vee\subseteq L_1\subseteq\cdots\subseteq L_{n-1}\subseteq P^\vee$
so that $L=L_{n-1}$ as the associated lattice to the centralizer of the
prior $n-1$ elements. From the definition of these lattices,
when they are quotiented out by the coroot lattice, they will either
be a cyclic subgroup of the center whose order divides the order of the 
center or will be the entire center. Thus we have the above conclusion
since $S\subseteq L/{Q^\vee}\subseteq P^\vee/{Q^\vee}.$ 
\end{proof}

\section{Properties of Centralizers} 
In general, the component group of an ordered $n$-tuple is some subquotient of the Weyl group and lies in the 
connected component $Z^0(x_1,\ldots,x_n).$  
In order to determine the component group for a non-simply connected
group we note that the finite diagonal subgroup contained in the center of each 
centralizer $Z(x_1,\ldots, x_k),$ for some $k,$ at some point 
becomes the component group and therefore defines the singularities
in the moduli space. For the classical groups, $Z(x_1)$ will be a product of type
$A_n,\ B_n$ or $D_n$ and for the exceptional groups, $Z(x_1,x_2)$ will be
of type $A_n,\ D_n$ Therefore, it sufficies to consider
the diagonal group action of the fundamental group $\pi_1(DZ(x_1))$  
on groups of these types. Since the fundamental group is a subgroup 
of the center of the simply connected covering, for type $B_n$ we
only consider the $\z/2\z$ action on the alcove given by flipping two vertices;the action of any higher
order central cyclic group is trivial. 
%Using this setup, we
%define the fixed spaces of groups of type $A_n$ and $D_n,$ generlize this to 
%subgroups of $G$ of the form $G_1\times_F G_2$ where $G_1, G_2$ are 
%subgroups of $G$ and $F$ is the finite central subgroup in the intersection of the 
%semisimple subgroups of $G_1$ and $G_2$. 

\begin{D}
Define the {\bf rank}
$rk(x_1,\ldots,x_n)$ of an $n$-tuple to be the rank of $Z(x_1,\ldots,x_n).$ An $n$-tuple has
{\bf rank zero} if and only if  $Z(x_1,\ldots,x_n)$ is a finite group.
A $c$-pair $(x,y)$ is in {\bf normal form} with respect to the maximal torus
$T$ in the alcove $A$ if $x\in T$ is the image under the exponential map of 
$\tilde{x}\in \mathfrak{t}^c$ and $y\in N_G(T)$ projects to $w_c\in 
\mathcal{W}.$ Note $w_c\in W$ is the differential action of $c\in\mathcal{C}G$ 
that, as a group of affine isometries of the Lie algebra $t$ of the maximal torus 
$T$ normalizes the alcove $A.$
\end{D}

Let $\tilde{\Delta} =\{ \tilde{a},a_1,\ldots, a_n\}$ be the set of
extended simple roots for a Lie group $G.$ Any closed
subset of the extended simple roots for $G$ gives a subdiagram of the
extended diagram. We are interested in the subset of roots $\Delta_x$ that annihilate the $n$-tuple. 
In particular, the simple root system for the centralizer $Z(x)$ for any $x\in G$ is defined as
by $\x=\{a\in \Delta\mid \a(x)\in\z\}$ which has an associated Weyl group  $\mathcal{W}(\Phi(x)).$ Let $\ov{x}=(x_1,\ldots,x_n)$ be a commuting $n$-tuple. Any element $x\in \pi_0(Z(\ov{x}))$ can be represented by $g\in 
Z_G(\ov{x})$ since $g$ normalizes $Z_G(\ov{x})$ and therefore via conjugation defines a 
map $\{x\}\times I \to G$ defining the path components of $G$.
Let $S_{\x}$ be the maximal torus in the centralizer
$Z(x_1,\ldots,x_n)$ generated by the roots in $\Delta_x$ given by its Lie algebra
$\mathfrak{s}= \underset{a\in\x}{\bigcap}{\rm Ker}(a).$
Since $G$ is reductive, there is a standard decomposition given by
$G=(\mathcal{C}G)^0\times_F DG$ where $(\mathcal{C}G)^0$ is a central 
torus in the semisimple subgroup $DG$ of $G.$ and 
$F=(\mathcal{C}G)^0\cap DG$ is a finite subgroup of the center of $DG.$  
Thus we get a decomposition of the the centralizer into
$Z(S_{\x})= S_{\x}\times_{F_{\x}} DZ(S_{\x}).$ 

\begin{remark}
For $c\in\mathcal{C}G$, denote by $S_c$ the torus in $T$ fixed under the 
action of the center. 
The choice of roots $\Delta(c)=\{a\in \Delta\mid r_a\not\in {\bf Z}\}$ 
where $c=\rm{exp}(\lambda)$ for $\lambda=\underset{a\in\Delta}{\sum}
r_a a$ defining the fixed subtorus $S_c\subset T$ is independent of the 
choice of lift $\lambda,$ let $\lambda'\underset{a\in\Delta}{\sum}
r_a'a'\in\mathfrak{t}$ be such that $\rm{exp}(\lambda')=c'.$ 
Then under $\rm{exp}\colon\mathfrak{t}\to T,$ the kernel of 
this map is an integral lattice defined with respect to $T.$ Namely,
$\rm{Ker}(exp) =Q^\vee,$ thus for $\lambda-\lambda'\in \rm{Ker}(\pi)$
this implies that $r_a-r_a'\in {\bf Z}$ and hence $\Delta(c)=\Delta(c')$
if and only if $r_a-r_a'\equiv 0\pmod{\bf Z}$ which we have since 
$r_a-r_a'\in Q^\vee.$ Therefore, $\Delta(c)$ depends only on the choice
$c\in\mathcal{C}G\cong P^\vee/Q^\vee$ and any two elements in the Lie
algebra $\mathfrak{t}$ differ by an element in the coroot lattice 
$Q^\vee.$ By definition of a $c-$pair of rank zero, $c\in DZ(S_c).$ 
Thus the moduli space is precisely $\mathfrak{M}=(T\times T)/W(T,G).$  
It certainly will not be true that for a general non-simply connected 
group that every element of a commuting $n$-tuple can be put inside the 
maximal torus. 
\end{remark}

We show that the fundamental group of the centralizer is finite cyclic 
by using diagram automorphisms. 
\begin{proposition}\label{fundgp}
Under a cyclic permutation of the vertices in the extended diagram 
of the type $A_n$ where the permutation is given by the fundamental
group $\pi_1(DZ(x_1))=\z_k,$ the quotient space has the form 
$$A_n/{\mathbb Z}_k\cong \underset{k}{\times}A_{l-1}\times T^{k-1}\rtimes
{\mathbb Z}_k$$
where $n+1=kl$ and $\pi_1(DZ(x_1))\cong {\mathbb Z}_k,$ with $1\le k\le 6.$ 
\end{proposition}
\begin{proof}
Since any inner automorphism of type $A_n$ is dihedral, it is either 
a rotation or a reflection. Consider the cyclic permutation $\tau$ given by 
rotation. (note: this is an element of a group of affine automorphisms
of a vector space which normalizes the alcove of a root system on that 
vector space. Such automorphisms are equivalent to diagram automorphisms
of the extended Dynkin diagram of the root system.)
If $\tau\in\z_k$ has order $n+1$ then the fixed point set
is simply the barycenter and thus $\mathfrak{t}^\tau=\{0\}.$ If $n$ 
is odd then $\tau$ may have order $k\mid(n+1).$ If $k=\frac{n+1}{2}$
then either the barycenter is the only fixed point or the fixed point
set is the join of type $A_{2k-1}$ or there is a rotation subgroup 
of $\tau$ of order exactly $k$ which implies it is an involution of 
the extended diagram which fixes two vertices and thus the quotient 
coroot diagram is a product of type $A_1.$ The $\z_2$ action on the 
alcove over $A_1$ is simply to switch the two vertices leaving the 
barycenter fixed. 

Specifically, if $n+1=kl$ then every node in the extended diagram included 
in this $k$-orbit is nonzero which leaves the quotient coroot diagram 
as the join of $k,\ (l-1)$-simplicies with the barycenter (since the 
barycenter is the only fixed space under the action of the full center)
times the remaining torus and semidirect product with rotation group. 
In terms of extended roots in the diagram, if 
$\tilde{\Delta}=\{\tilde{a},a_1,\ldots,a_n\}$ is the set of simple roots
for $A_n$ then the quotient space $A_n/\z_k$ is defined by the elements
in the orbit, 
$\tilde{\Delta}/{\z_k}=\{\tilde{a},a_l,a_{2l},\ldots,a_{(k-1)l}\}.$ 
Thus the gaps between the nodes are of length $(l-1).$ 
Therefore, the fixed space will be given by
$A_n/{\z_k}= (\underset{k}{\times}A_{l-1})\times T^{k-1}\rtimes\z_k.$
What we have shown is that in $A_n$ the 
$\rm{Stab}_{\tau}(\underset{k}{\times}A_{l-1})=\z_k.$ 
The fact that $\pi_1(DZ(x_1))\cong \z_k$ where $1\le k\le 6,$ follows 
directly from looking at the coroot integers for all the extended 
Dynkin diagrams. 
\end{proof}
 
\begin{proposition}
Let $G$ be a simple group of dimension $n.$ The centralizer 
$Z(A_k\times_F T^{n-k})= Z_{A_k}\times_F T^{n-k}.$
\end{proposition}
\begin{proof}
Given $A_k\times_F T^{n-k},$ for any element $[A,t]=[A\zeta,\zeta^{-1}A]\in A_k\times_F T^{n-k}$
its centralizer is $Z([A,t])=\{[B,s]\colon [A,t][B,s]=[B,s][A,t]\}.$ This 
implies that $[AB,ts]=[BA,st].$ But since $st=ta\in T,\ AB=BA.$ 
Therefore, $Z([A,t])=Z_{A_k}(A)\times_F T^{n-k}$ which is 
connected and thus Proposition ~\ref{fundgp} applies. The conclusion 
follows because the components in the almost direct product are 
simply connected. 
\end{proof}

\begin{proposition}\label{basic}
Given $G_1\times_F G_2$ where $G_1,G_2$ are subgroups of $G,\
F\subseteq \mathcal{C}G_1$ and $F\subseteq\mathcal{C}G_2$ and 
$F\in(DG_1\cap DG_2).$ Then for $[a,b]\in G_1\times_F G_2,$ 
$$\{1\} \to F\to Z([a,b])\overset{\pi}{\rightarrow} Z_{G_1}(a)\times_F 
Z_{G_2}(b)\rightarrow F.$$
\end{proposition}
\begin{proof}
Consider the map 
$G_1\times_F G_2\overset{\pi}{\rightarrow} G_1/F\times G_2/F.$ 
The kernel is $\rm{ker}(\pi)=\{[c,d]\colon c,d\in F\}.$ Thus we 
have the injective map 
$\{1\}\to G_1\times_F G_2\overset{\pi}{\rightarrow} G_1/F\times G_2/F.$ By definition of the 
centralizer of an element in $G_1\times_F G_2$  
\begin{eqnarray*}
Z([a,b]) &=&\{[c,d]\colon [ac,bd]=[ca,db]\} \\ 
&=& \{\exists f_1,f_2\in F,\
[acf_1,bd]\sim [ac,f_1^{-1}bd]=[caf_2,db]\sim [ca,f_2^{-1}db]\} \\
&=&\{[c,d]\colon [a,c]=f,[b,d]=f^{-1},\ f=f_1f_2^{-1},\ [a,c][b,d]=1\}
\end{eqnarray*}
This demonstrates that the coker of $\pi$ is $F$ and that $\pi$ 
is not surjective. Therefore, 
$\frac{\pi^{-1}(Z([a])\times Z([b]))}{F}=Z([a,b]).$
Note also that by the definition of the centralizer of 
$[a,b]\in G_1\times_F G_2,$ that the generalized Stiefel-Whitney 
class \cite{M00} is $w_2(a,c)= -w_2(b,d)\in F.$ Hence $w_2\colon H^2(T^*)\to {\bf Z}_n$
defines an obstruction.  
\end{proof}

\begin{corollary}
Following propsition ~\ref{basic}, if 
$G_1=T$ for some torus and $G_2$ is of type $A_r$ then 
$$F\to Z_{A_r}(\tilde{A})\overset{\pi}{\rightarrow} Z_{A_r/F}([A])
\rightarrow F\to \{1 \}.$$
\end{corollary}
\begin{proof}
Given a sequence 
$\{1\}\to T^k\hookrightarrow T^k\times_F A_r\overset{\pi}{\rightarrow}
A_r/F\,$ inside $Z_G([t,A])$ we have that 
$[t,A][s,B]=[ts,AB]=[st,BA]=[s,B][t,A]$ and in 
$Z_{A_r}([A])=\{[B]\colon [A,B]=[B,A]\}$ which implies that 
$Z_G([t,A])\hookrightarrow \pi^{-1}(Z_{A_r}([A]))$ so that $AB=BA\zeta$
for $\zeta\in F.$ Thus they are equal up to an element in the finite 
group. Therefore we have 
$Z_G([t,A])\hookrightarrow \pi^{-1}(Z_{A_r}([A]))\to F.$ 
Suppose that $[A]\in A_r/F$ and consider its lift $\tilde{A}\in A_r$ 
arbitrary. Then 
$$F\to Z_{A_r}(\tilde{A})\overset{\pi}{\rightarrow} Z_{A_r/F}([A])
\rightarrow F\to \{1 \}$$ because the kernel is
$\rm{Ker}(\pi)=F$ and from what we have already deduced, $AB=BA\zeta$
for $\zeta\in F$. Hence 
$$Z_G([t,A])=\frac{\pi^{-1}(Z_{A_r}([A]))}{F}.$$  
We used the simply connected component as follows. If we consider 
$[B]$ such that there exists a $\tilde{B}$ with 
$\tilde{A}\tilde{B}=\tilde{B}\tilde{A},$ then mulitiplication of the 
equivalence classes is 
$[s,\tilde{B}][t,\tilde{A}]=[st,\tilde{B}\tilde{A}=[ts,\tilde{A}\tilde{B}].$
Since $\pi_1(G)=\{ 1\}$ when we lift to the universal covering we can say that for
$[t,A]\in \tilde{G}=T^k\times_F G$ then $Z_{\tilde{G}}([t,A])=
T^k\times_F Z(A)$ and more importantly that $Z_{\tilde{G}}([t,A])$ is 
connected. 
\end{proof}

\begin{corollary}
Consider a subgroup in $G$ of the form $A_k\times_F A_r,$ for $r+k=n+1,$
then the centralizer of an element $[a,b]\in A_k\times_F A_r,$ 
for $r+k=n+1,$ is 
$Z_G([a,b])=\frac{(Z_{A_k}(a)\times_{F'} Z_{A_r}(b))}{F},$ \ where \ 
$F'=\mathcal{C}DZ_{A_k}\cap\mathcal{C}DZ_{A_r}\supseteq F.$  
\end{corollary}

\begin{corollary}
Consider a subgroup in $G$ of the form $A_k\times_F D_{n-k}.$
then the centralizer of an element $[a,b]\in A_k\times_F D_{n-k}$
is of the form 
$$Z_G([a,d])=\frac{(Z_{A_k}(a)\times_{F'} Z_{AD_{n-k}}(d))}{F},$$ where
$F'=\mathcal{C}DZ_{A_k}\cap \mathcal{C}DZ_{D_{n-k}}\supseteq F.$  
\end{corollary}

It does not necessarily follow that 
$\pi_0(Z(x_1,\ldots,x_n)\subset\pi_1(DZ(x_1))$ because $DZ(x_1)$ is 
not necessarily connected. The fact that for $G$ of type $D_n$ that
$\pi_1(D_n)=\mathcal{C}D_n\cong\z\times\z$ and that the 
characteristic class for a principal $G$-bundle over $T^n$ lies in
$H^2(T^n;\pi_1(G))\cong\z\times\z$ means that there is a 
possibility that the component group for an $n$-tuple inside $D_n$ 
will not be finite cyclic. 

\section{Uniform Bound on Chains of Commuting Elements}

Assume that $G$ is a compact, connected Lie group.  
Define a chain of elements $x_1,x_2,x_3,\ldots\in G$ where
$x_i\in Z_G(x_1,\ldots,x_{i-1})$ to be {\sl redundant} if for any
$x_i\in \mathcal{C}Z_G(x_1,\dots,x_{i-1})$ then
$Z_G(x_1,\ldots,x_i)=Z_G(x_1,\ldots,x_{i-1}).$ Otherwise, the
chain is said to be {\sl irredundant}.
%the chain does not have to of finite length!
Define an {\sl ordering} of the $n$-tuple by the property
$\rm{dim}\ Z(x,x)\ge \rm{dim}\ Z(x,x,y)$ for $x\ne y.$
Thus for a chain of commuting elements $x_1,x_2,\ldots\in G$
if we consider a decreasing chain of the centralizers of these
elements,
$$G\supset Z(x_1)\supset Z(x_1,x_2)\supset\cdots\supset
Z(x_1,x_2,\ldots,x_n)=Z(x_1,x_2,\ldots,x_n,t_1,\ldots t_k),$$
for $t_i\in CG$, the chain of strict inclusions will terminate
once the chain of elements becomes redundant, for any further
choices of $x_i\in G.$. This is equivalent to saying that the
chain of decreasing centralizers will terminate once  the
centralizer of an $n$-tuple $Z(x_1,\ldots,x_n)$
becomes abelian; this will occur if the centralizer is either the
maximal torus or a finite group. It is clear that if $x$ is a regular
element then $Z(x)=T$ and the chain terminates immediately.
Define the {\bf cardinality} $m$ of the centralizer
to be $m(Z(x_1,\ldots,x_n))=n$ where $n$ is the maximal length of an
irredundant chain.To compute $m(G)$ for $G,$ we determine the centralizer $Z_G(x_1,\ldots,x_n)$ for different 
commuting $n$-tuples
in some numerical manner and relate the centralizer to certain nodes
in a subdiagram of the extended Dynkin diagram, where the subdiagram
is given by a closed subset $I\subseteq \tilde\Delta$ of roots which
annihilate all the elements in the commuting $n$-tuple. 
Define a {\bf c-tuple} to be an $n$-tuple $(x_1,\ldots,x_n)$
such that $[x_1,x_2]=c$ and $[x_i,x_1]=[x_i,x_2]=1$ for $i>2.$

Fix a maximal torus $T$ in $G$ and let $A$ the alcove over the torus.
Then there is an affine linear transformation of the vector space 
$\mathfrak{h}$ defined as follows: 
$\phi\colon A\to A$ by $w_c(t-\zeta_{c^{-1}})$ 
where $t\in\mathfrak{h},\ w_c\in\mathcal{W}$ is the projection onto 
the Weyl group from an element $y$ in a $c$-pair $(x,y)$ which normalizes 
the centralier $Z(x),$ and $\rm{exp}(\zeta_{c^{-1}})=c^{-1}\in 
\mathcal{C}G.$ The fixed space of this automorphism of the alcove 
is the subspace $A^c=A\cap\mathfrak{h}^{w_c}.$ Since the alcove $A$
is the fundamental domain and clearly $A^c$ contains the barycenter,
then $A^c$ lies in the connected component of the identity. 

\begin{lemma}
Let $G$ be a compact, connected Lie group of rank $n$ and 
$A\in\mathfrak{h}$ be
the alcove over the maximal torus. Then for any $\tilde{x}\in A^c$ such 
that $\rm{exp}(\tilde{x})=x\in T,$ it is sufficient to consider that 
either $\tilde{x}$ is a vertex in the fixed subspace $A^c$ under the 
central action of $c\in\mathcal{C}\tilde{G}$ or that $\tilde{x}$ lies on 
an edge
whose vertices are both central elements in order to get a subgroup of
maximal rank in $G.$ 
\end{lemma}
\begin{proof}
The case when $G$ is of type $A_n$ is unique in that the fixed space
under the central action is the barycenter. Thus the only choice for 
$\tilde{x}\in A^c$ is the barycenter and its centralizer is given by
$\tilde{Z}(x)=\mathfrak{h}\rtimes\mathcal{C}\tilde{G}.$ 

Every $c$-pair
is conjugate to one in normal form which means that $\tilde{x}$ 
corresponds to the barycenter in $A^c$ and $\tilde{y}$ projects to an 
element in the Weyl group $\mathcal{W}.$ Since a $c$-tuple is defined
by a $c$-pair and then $n-2$ commuting elements, we may always choose
$\tilde{x}$ to be the barycenter in $A^c$ and therefore it will lie
in the connected component of the identity. Thus it follows from the 
commuting case that the maximal subgroup always corresponds to choosing
$\tilde{x}$ as a vertex in $A^c.$ 

As a result of the commuting case, it's necessary to choose 
$\tilde{x}\in A^c$ which is connected. However, when working with
$\langle\mathcal{C}\rangle$-tuples, we may choose $\tilde{x}$ to lie
either in $A^c$ or not. Elements in the connected component of the 
identity of the centralizer follow from the case for commuting $n$-tuples.
 The center $\mathcal{C}\tilde{G}$ is a finite group, 
thus the key is to always follow and extremal path when choosing elements 
in the fixed space $A^c.$ 
\end{proof}

\begin{corollary}
Let $\ov{x}=(x_1,\ldots,x_n)$ be an $n$-tuple for a compact, connected Lie 
group $G$. Assume that the $n$-tuple is 
irreducible. Then it is sufficient to only consider the semisimple part 
when determining the moduli space. 
\end{corollary}
Take $Z_G(\ov{x})$ and look at $S_{max}\subseteq Z_G(\ov{x}).$ Then take 
the centralizer of this small maximal torus
and consider its maximal torus $S\subseteq Z_G(S_{max}).$ Then 
$Z_G(S_{max})=S_{max}$ which implies that $S_{max}$ 
is semisimple because it's a torus so communtes with the maximal torus $S$. 
Thus there exists a $g\in Z_G(S_{max})$ 
such that $gS_{max}g^{-1}=S.$ But by definition  $S=S_{max}.$ 

The following theorem states that there exists a uniform upper bound
on irredundant chains of commuting elements in $G.$
\begin{theorem}\label{Weaktheorem} (Weak Theorem)
Given a compact Lie group $G$ there exists an integer $m=m(G)$ 
such that any irredundant chain has fewer than $m$ elements.    
\end{theorem}

\begin{theorem}\label{Strongtheorem} (Strong Theorem)
Given $G$ a compact Lie group, there exists a finite number
of subgroups $H_1,\ldots,H_k\subset G$ such that for any $x\in G,$
$Z_G(x)$ is conjugate to $H_i$ for some $i.$
\end{theorem}
\begin{proof} (Strong Theorem) 
Since $G$ is compact, there is a decomposition
$G=(\mathcal{C}G)^0\cdot_F DG$ where $F=(\mathcal{C}G)^0\cap DG$
is a finite group. Choosing $x\in (\mathcal{C}G)^0$ then $Z_G(x)=G$ 
hence it suffices to choose $x\in DG.$ If $x\in T$ is a non-central 
element for a choice of a maximal torus $T\subset G,$ then since $T$ is
abelian and $x$ is regular, $Z(x)=T$ is the smallest possible
subgroup of maximal rank. We claim that there are a finite
number of closed subgroups between $G$ and $T$ such that each is
of the form $Z(x_i)$ for some $x_i\in G.$ 

When $\tilde{x}$ corresponds to a vertex
or an edge with two central vertices, the element yields a maximal rank
subgroup. Thus choose $x_1$ as image under the exponential map of 
a non-central vertex. Then $H_1=Z(x_1)=DZ(x_1)\cdot\mathcal{C}^0Z(x_1).$
If $Z(x_1)$ is connected 
then we may choose $\tilde{x_2}$ as a vertex or a "central edge" which
gives $H_2=Z(x_1,x_2)$ where the rank of the semisimple part of the 
centralizer decreases by at least one. In other words, continuing to 
choose vertices or "central edges" such that the rank decreases, we
are following a connected path and at some point there will be some
$H_k$ such that the rank of the semisimple part of $H_k$ is zero. 
Thus $H_k$ is abelian and we have proven the claim. 

If $Z(x_1)$ is not connected, then we may restrict to the connected
component of the identity which is still reductive and the above 
analysis holds as long as there is no choice of element $x_j$ which 
is fixed by the finite abelian subgroup of the center of 
$H_{j-1}=Z(x_1,\ldots,x_{j-1}).$ Otherwise, if some $x_j\in 
Z(x_1,\ldots,x_{j-1})$
is fixed by this diagonal subgroup $\Delta$ then 
$H_j=Z(x_1\ldots,x_j)=DZ(x_1,\ldots,x_j)\cdot \mathcal{C}Z\rtimes\Delta$ 
is disconnected. If the semisimple part is not abelian, then we must 
choose elements in the semisimple part until it becomes abelian. At 
the point in which it does, say some $H_{k-1}$ where $1\le j < k-1,$
choosing $x_k\in\Delta$ makes $H_k$ into a finite abelian group and 
thus we have again shown that there are only a finite number of 
normal subgroups $H_1,\ldots,H_k$ in $G$ which are conjugate to the 
centralizer of some number of elements. 

$H_1=DZ(x)$ is a normal simple subgroup in $DG$ which is characteristic
in $DG$ because it lies in the center of $DG$ and the center is always
characteristic for any group. $H_1$ is also closed since it is a semisimple subgroup 
in $G$ compact. If $x$ is a semisimple element, then it is
conjugate to an element in the torus $T$ which implies that the 
connected component $Z^0(x)$ contains a torus $T'$ which is 
conjugate to $T.$ In other words,
$Z^0(x)=DZ^0(x)\cdot\mathcal{C}Z^0(x)$ where the
semisimple part corresponds to a subsystem subgroup and the other
is a commuting torus. To consider the remaining cases, we use the algorithm by
Borel and de Siebenthal (Theorem 3 \cite{BS49}) which determines all the closed subsystems
of $\Phi(G)$ (all the roots of $G$) for closed
subgroups of maximal rank. Take the extended
Dynkin diagram $\tilde{D}(G)$ and remove a finite collection of
nodes, repeating the process with all the connected components of the
resultant graph. These diagrams correspond to subdiagrams of the
subroot systems and there is one conjugacy class of subsystem   
subgroups for each such subsystem. In particular, if $G$ is connected,
compact then it possesses only finitely many conjugacy classes of subsystem
subgroups and hence the theorem is satisfied. It remains to prove what happens if $G$ is disconnected 
because a direct product of connected subgroups will each
have a finite number of subgroups each conjugate to $Z(x_i)$ for some
$x_i\in G.$ Note that there does not necessarily have to be any ordering on
these $H_i.$ In fact, it would be virtually impossible to impose
such a requirement because the $H_i$ are not necessarily 
subgroups of each other. 
If $G$ is the semidirect product of a connected group $U$ and a discrete
group $V$ then the above applies to the connected part and the only
centralizer for any of the elements in the discrete subgroup is the
entire discrete group. Thus there is a still finite number. If 
$G$ is disconnected, then there is a finite number of subgroups of 
type $Z_G(x)$ when restricting to the connected component of the 
identity. 

If $G$ is a disconnected group where both the clopen sets $U,V$
are connected, then there is still a finite number. If $U$ is connected and
$V$ is finite and disconnected, then count the number of such $H_i$ in each
of the smaller connected sets, ultimately leaving a disconnected set with exactly one point whose
centralizer is itself, since it is totally disconnected.
If $U$ is also totally disconnected there are a finite number of points $u_i\in U$ because $G$ is compact and 
thus $H_i$ will be conjugate to the  $u_i=Z(u_i).$ 
If $G=U\rtimes V$ where both $U,V$ are connected then there
exist $H_1,\ldots, H_l$ for $U$ and $K_1\ldots, K_r$ for $V$
such that each $H_i$ is conjugate to $Z(u_i)$ and $K_j$ is
conjugate to $Z(v_j),\ 1\le i\le l,\ 1\le j\le r.$
\end{proof}

For a simply connected, connected Lie group, the following
shows how theorem ~\ref{Strongtheorem} implies 
theorem ~\ref{Weaktheorem}.
 
\begin{proposition}
If $G$ is a compact, connected, simply connected Lie group and
there are a finite number of subgroups $H_1,\ldots,H_k\subset G$  
such that $Z_G(x)\cong H_i$ for some $i$ and for any $x\in G,$
then there is an integer $m$ such that any irredundant chain has
fewer than $m$ elements.
\end{proposition}
\begin{proof}
Since $G$ is connected, it is sufficient to choose $x\in T$
such that
$x={\rm exp}(\tilde{x})$ where $\tilde{x}$ is in the alcove $A.$
Since we are only dealing with commuting elements in the chain,  
we are concerned with what happens in the walls of the alcove.
The proof is by induction. Define a pair of numbers
$(d,n)$ lexicographically where $d$ represents the dimension of
the subgroup and $n$ represents the number of components.
It is vacuously true for dimension one (we could choose $S^1$
if we wanted to make it slightly less trivial since $S^1=T^1$ is
the $1$-dimensional torus in the Lie group). Suppose that the
Weak theorem is true, i.e. that for every $H\subsetneqq G$ there
exists and $m$ for any irredundant chain. Note that if we let 
$(x_1,\ldots,x_k)$ be an irredundant chain in $G$ such 
that $Z_G(x_1)=H,\ x_1\not\in\mathcal{C}G,$ then since we have a
chain of commuting elements, $(x_2,\ldots,x_k)$ is an irredundant
chain in $H$ which, by induction, implies that $k\le m(H) + 1.$ Hence
any irredundant chain in $G$ has length
$m \le 1 + {\rm max}(m(H))$ where the maximum is over the
subgroups $H$ strictly contained in $G$ and where $H=Z_G(x)$ for
some element $x\in G.$  

Now consider the connected component of the identity, $H^0\subset G^0.$
By the lexicographic ordering, ${\rm dim}(H^0)\le {\rm dim}(G^0)$
if either ${\rm dim}(H^0)<{\rm dim}(G^0)$ or their dimensions are
equal and $n\le n'.$ But
${\rm dim}(H^0)={\rm dim}(G^0)$ if and only if $H^0=G^0.$ So
$G^0\subset H\subsetneqq G$
which implies that the components $\pi_0(H)\underset{\neq}{<}\pi_0(G).$
Thus we have shown that the Strong theorem implies the Weak theorem.
\end{proof}

\def\cprime{$'$}
\def\p{{\mathbb P}}

%\bibliographystyle{alpha}
%\bibliography{almostcommuting2009}

\begin{thebibliography}{10}

\bibitem{A57}
M. Atiyah
\newblock {\em Vector Bundles over elliptic curves}
\newblock Proc. London Math. Soc {\bf 7} (1957), 414-452.

\bibitem{AB82}
M. Atiyah and R. Bott
\newblock {\em The Yangs-Mills equations over Riemann surfaces}, 
\newblock Phil.Trans.Roy.Soc. London A  {\bf 308} (1982), 523-615.

\bibitem{B62}
A. Borel
\newblock {\em Sous-groupes commutatifs et torsion des groupes de Lie compactes}
\newblock T$\hat{o}$hoko Math.Jour. {\bf 13} (1962), 216-240.

\bibitem{BFM02}
A. Borel, R.D. Friedman, and J.W. Morgan
\newblock {\em Almost commuting elements of compact Lie groups}
\newblock Mem.Amer.Math.Soc. {\bf 157} (2002), no. 747, x+136.

\bibitem{BS49}
A. Borel and J. De Siebenthal
\newblock {\em Les sous-groupes ferm\'{e}s de rang maximum des groupes de Lie closs}
\newblock 

\bibitem{Bourbaki}
N. Bourbaki
\newblock{\em Groupes et Algebres de Lie}
\newblock Chap. 4,5, et 6, Masson, Paris, 1981.


\bibitem{FMW98}
R.D. Friedman, J.W. Morgan, and E. Witten
\newblock {\em Principal $G$-bundles over elliptic curves}
\newblock Math.Res.Lett. {\bf 5} (1998), 97-118.

\bibitem{H95}
J. Humphreys
\newblock{\em Conjugacy classes in semisimple algebraic groups}
\newblock Math.Surv.Monogr.,vol 43, Amer.Math.Soc., Providence, 1995.

\bibitem{KS00}
V. Kac and A. Smilga
\newblock {\em Vacuum structure in supersymmetric Yang-Mills theories with any gauge group}
\newblock hep-th/9902029, v. 3. 

\bibitem{K99}
A. Keurentjes
\newblock {\em Non-trivial flat connections on the $3$-torus}
\newblock hep-th/9901154, J.High Energy Phys. {\bf 9905} (1999), 001 (electronic). 

\bibitem{L76}
E. Looijenga
\newblock {\em Root systems and elliptic curves}
\newblock Invent. Math. {\bf 38} (1976), 17-32. 

\bibitem{M00}
B. McInnes
\newblock{\em Gauge spinors and string duality}
\newblock Nuclear Physics B {\bf 577} (2000), 439-460.

\bibitem{S75}
R. Steinberg
\newblock {\em Torsion in reductive groups}
\newblock Advances in Math. {\bf 15} (1975), 63-92. 

\bibitem{W98}
E. Witten
\newblock {\em Toroidal compactification without vector structure}
\newblock J. High Energy Phys. {\bf 9802} (1998), 006 (electronic) 

\end{thebibliography}
\end{document}